\documentclass{amsproc}

\usepackage{amscd, amsmath, amsthm, amssymb, mathtools, mathrsfs, stmaryrd}
\usepackage{ascmac}
\usepackage{comment}
\usepackage{bm, bbm}
\allowdisplaybreaks

\usepackage{graphicx}
\usepackage{fullpage}

\usepackage{hyperref}

\usepackage{tikz}
\usetikzlibrary{intersections, calc, arrows.meta}

\usepackage{here}
\usepackage{time}
\usepackage[abbrev]{amsrefs}

\usepackage{xcolor}
\usepackage[capitalize,nameinlink,noabbrev,nosort]{cleveref}
\hypersetup{
	colorlinks=true,       
	linkcolor=brown,          
	citecolor=brown,        
	filecolor=brown,      
	urlcolor=brown,           
}

\makeatletter
\@namedef{subjclassname@2020}{\textup{2020} Mathematics Subject Classification}
\makeatother

\newtheorem*{theorem*}{\hspace{-6.3mm}\textbf{Theorem}}  

\newtheorem{theoremcounter}{Theorem Counter}[section]

\theoremstyle{remark}

\theoremstyle{definition}

\theoremstyle{plain}

\newtheorem{corollary}[theoremcounter]{Corollary}

\newtheorem{theorem}[theoremcounter]{Theorem}

\numberwithin{equation}{section}

\newcommand{\Z}{\mathbb{Z}}
\newcommand{\Q}{\mathbb{Q}}

\newcommand{\C}{\mathbb{C}}

\newcommand{\dd}{\mathrm{d}}

\newcommand{\scB}{\mathscr{B}}

\begin{document}

\title[]{Applications of Fa\`{a} di Bruno's formula to partition traces}

\author[]{Toshiki Matsusaka}
\address{Faculty of Mathematics, Kyushu University, Motooka 744, Nishi-ku, Fukuoka 819-0395, Japan}
\email{matsusaka@math.kyushu-u.ac.jp}




\maketitle

\begin{abstract}
	We revisit several partition-theoretic generating functions, including the theta quotients from Ramanujan's lost notebook, MacMahon's partition functions, and reciprocal sums of parts in partitions, through the lens of the classical Fa\`{a} di Bruno formula.
	This approach offers a unified and natural reinterpretation of known results and provides a systematic framework for deriving new identities of a similar type.
\end{abstract}


\section{Introduction}\label{sec:Intro}

A \emph{partition} of a positive integer $k$ is a nonincreasing sequence of positive integers
\[
	\lambda = (\lambda_1, \lambda_2, \dots, \lambda_r)
\]
such that the sum of the parts is $k$. We write this as $\lambda \vdash k$. Equivalently, a partition can be expressed in frequency notation as $\lambda = (1^{m_1}, \dots, k^{m_k}) \vdash k$, where each $m_j \ge 0$ denotes the multiplicity of the part $j$. The \emph{length} of $\lambda$ is $\ell(\lambda) \coloneqq m_1 + \cdots + m_k$. For each $\lambda = (1^{m_1}, \dots, k^{m_k}) \vdash k$, we associate the monomial
\[
	X_\lambda \coloneqq \prod_{j=1}^k X_j^{m_j}.
\]
To introduce the notion of partition traces, let $\phi: \mathcal{P} \to \C$ be a map on the set $\mathcal{P}$ of all partitions. For each positive integer $k$, the \emph{partition trace} of $\phi$ is defined by
\[
	\mathrm{Tr}_k(\phi; X_1, \dots, X_k) \coloneqq \sum_{\lambda \vdash k} \phi(\lambda) X_\lambda.
\]
For instance, when $\phi$ is the constant map $\phi(\lambda) = 1$, the trace $\mathrm{Tr}_4(1; X_1, \dots, X_4)$ is given by
\[
	\mathrm{Tr}_4(1; X_1, X_2, X_3, X_4) = X_4 + X_3 X_1 + X_2^2 + X_2 X_1^2 + X_1^4.
\]

A particularly interesting case arises when each $X_j$ is specialized to the Eisenstein series $E_{2j}(q)$. The resulting series, called the \emph{traces of partition Eisenstein series}, have been studied in several recent papers. Here, the \emph{Eisenstein series} $E_{2k}(q)$ is defined by
\[
	E_{2k}(q) \coloneqq 1 - \frac{4k}{B_{2k}} \sum_{n=1}^\infty \sigma_{2k-1}(n) q^n,
\]
where $B_k$ is the $k$-th Bernoulli number (with $B_1 = 1/2$, to be used later), and $\sigma_{2k-1}(n) \coloneqq \sum_{d \mid n} d^{2k-1}$ is the usual divisor sum. 
As a representative example, Amdeberhan--Ono--Singh~\cite{AmdeberhanOnoSingh2024-pre} studied the following quotient recorded in Ramanujan's Lost Notebook~\cite[p.~369]{Ramanujan1988}:
\[
	V_{2k}(q) \coloneqq \frac{\sum_{n \in \Z} (-1)^n (6n+1)^{2k} q^{\frac{n(3n+1)}{2}}}{\sum_{n \in \Z} (-1)^n q^{\frac{n(3n+1)}{2}}},
\]
and obtained the following explicit formula that refines earlier results of Berndt--Yee~\cite{BerndtYee2003}, (see also~\cite[Chapter 14]{LostNoteII}).

\begin{theorem*}[Amdeberhan--Ono--Singh]
	For partitions $\lambda = (1^{m_1}, \dots, k^{m_k}) \vdash k$, we define
	\[
		\phi_V(\lambda) \coloneqq 4^k (2k)! \prod_{j=1}^k \frac{1}{m_j!} \left(\frac{(4^j - 1) B_{2j}}{(2j) (2j)!}\right)^{m_j}.
	\]
	Then we have
	\[
		V_{2k}(q) = \mathrm{Tr}_k(\phi_V; E_2(q), E_4(q), \dots, E_{2k}(q)).
	\]
\end{theorem*}

Historically, partition traces have appeared in a variety of contents, not limited to the case where each $X_j$ is evaluated at the Eisenstein series. As another example, let us revisit Lehmer's 1966 work~\cite{Lehmer1966} on cyclotomic polynomials. Let $\Phi_n(x)$ be the $n$-th cyclotomic polynomial. In this work, Lehmer showed that the higher derivatives of $\Phi_n(x)$ at $x=1$ can be expressed in terms of partition traces as follows. 

\begin{theorem*}[Lehmer]
	For partitions $\lambda = (1^{m_1}, \dots, k^{m_k}) \vdash k$, we define
	\[
		\phi_\Phi(\lambda) = k! \prod_{j=1}^k \frac{(-1)^{m_j}}{m_j! j^{m_j}}.
	\]
	Then, for $n \ge 2$, we have
	\[
		\frac{\Phi_n^{(k)}(1)}{\Phi_n(1)} = \mathrm{Tr}_k(\phi_\Phi; \varsigma_1(n), \varsigma_2(n), \dots, \varsigma_k(n)),
	\]
	where we define
	\[
		\varsigma_k(n) = -\frac{1}{(k-1)!} \sum_{m=1}^k \frac{B_m}{m} s(k,m) J_m(n),
	\]
	with $s(k, m)$ denoting the Stirling number of the first kind and $J_k(n)$ the Jordan totient function.
\end{theorem*}

The striking similarity between these two results is noteworthy and suggests the presence of a common and more fundamental underlying structure. The purpose of this article is to show that a wide range of results can be understood in a unified and transparent way through the classical \emph{Fa\`{a} di Bruno formula}. 

Roughly speaking, Fa\`{a} di Bruno's formula provides an explicit expression for the derivative of a composite function. We refer the reader to \cite{Johnson2002, Craik2005} for the historical background and detailed proofs, as well as the recent article~\cite{Sullivan2022}. To be precise, we introduce the (complete) \emph{Bell polynomial} $\scB_k(X_1, \dots, X_k) \in \Z[X_1, \dots, X_k]$, defined by the generating function
\begin{align}\label{def:Bell-polynomial}
	\sum_{k=0}^\infty \scB_k(X_1, \dots, X_k) \frac{t^k}{k!} = \prod_{j=1}^\infty \exp \left(X_j \frac{t^j}{j!}\right).
\end{align}
The first few examples are
\begin{align*}
	\scB_1(X_1) &= X_1,\\
	\scB_2(X_1, X_2) &= X_2 + X_1^2,\\
	\scB_3(X_1, X_2, X_3) &= X_3 + 3X_2 X_1 + X_1^3,\\
	\scB_4(X_1, X_2, X_3, X_4) &= X_4 + 4X_3 X_1 + 3X_2^2 + 6X_2 X_1^2 + X_1^4.
\end{align*}
It is easy to see that Bell polynomials can be expressed via partition traces:
\[
	\scB_k(X_1, \dots, X_k) = \mathrm{Tr}_k(\phi_\mathrm{B}; X_1, \dots, X_k),
\]
where we set
\[
	\phi_\mathrm{B}(\lambda) = k! \prod_{j=1}^k \frac{1}{m_j! (j!)^{m_j}}.
\]

\begin{theorem}[Fa\`{a} di Bruno's formula]
	Let $k$ be a positive integer. Suppose $f(x), g(x)$, and $h(x)$ are functions with all necessary derivatives defined. For partitions $\lambda = (1^{m_1}, \dots, k^{m_k}) \vdash k$, we define
	\[
		\phi_\mathrm{F}(\lambda) = k! f^{(\ell(\lambda))}(g(x)) \prod_{j=1}^k \frac{1}{m_j! (j!)^{m_j}}.
	\]
	Then we have
	\begin{align*}
		\frac{\dd^k}{\dd x^k} f(g(x)) &= \mathrm{Tr}_k(\phi_\mathrm{F}; g^{(1)}(x), g^{(2)}(x), \dots, g^{(k)}(x)).
	\end{align*}
	In particular, by applying this formula with $f(x) = \exp(x)$ and $g(x) = \log h(x)$, we obtain
	\[
		\frac{h^{(k)}(x)}{h(x)} = \scB_k ((\log h)^{(1)}(x), (\log h)^{(2)}(x), \dots, (\log h)^{(k)}(x)).
	\]
\end{theorem}

The observation that Fa\`{a} di Bruno's formula can be applied to derive Lehmer's theorem is due to Herrera-Poyatos and Moree~\cite{HerreraMoree2021}. Indeed, it is easy to see that
\[
	\lim_{x \to 1} \frac{\dd^j}{\dd x^j} \log \Phi_n(x) = -\sum_{\substack{0 < k < n \\ (k, n) = 1}} \frac{ (j-1)!}{(e^{2\pi ik/n} - 1)^j},
\]
which matches $-(j-1)! \varsigma_j(n)$ as shown in \cite[Theorem 2]{Lehmer1966}. Therefore, Fa\`{a} di Bruno's formula implies that
\begin{align*}
	\frac{\Phi_n^{(k)}(1)}{\Phi_n(1)} &= \scB_k(-\varsigma_1(n), \dots, -(k-1)! \varsigma_k(n))\\
		&= \mathrm{Tr}_k(\phi_\Phi; \varsigma_1(n), \dots, \varsigma_k(n)).
\end{align*}

In the following sections, we provide further reinterpretations of the theorem of Amdeberhan--Ono--Singh, as well as several results on MacMahon's partition functions, their generalizations, and reciprocal sums of parts in integer partitions, all viewed through the lens of the Fa\`{a} di Bruno formula.

\section*{Acknowledgement}

This work was motivated by Byungchan Kim's talk at a conference ``Number Theory in the spirit of Ramanujan and Berndt", held in June 2025 at Yonsei University, South Korea, and by discussions on the topics covered in \cref{sec:SRP}. The author is also grateful to him for his helpful comments on an earlier version of this article. The author was supported by JSPS KAKENHI (JP21K18141 and JP24K16901) and the MEXT Initiative through Kyushu University's Diversity and Super Global Training Program for Female and Young Faculty (SENTAN-Q).

\section{Amdeberhan--Ono--Singh's theorem}

First, we reinterpret the original proof by Amdeberhan--Ono--Singh~\cite{AmdeberhanOnoSingh2024-pre} using Fa\`{a} di Bruno's formula. By Euler's pentagonal number theorem, we have
\[
	\eta(\tau) \coloneqq q^{1/24} \prod_{m=1}^\infty (1-q^m) = q^{\frac{1}{24}} \sum_{n \in \Z} (-1)^n q^{\frac{n(3n+1)}{2}} = \sum_{n \in \Z} (-1)^n q^{\frac{(6n+1)^2}{24}},
\]
where $q = e^{2\pi i\tau}$. Applying Fa\`{a} di Bruno's formula with respect to $\tau$, we obtain
\[
	\left(\frac{2\pi i}{24}\right)^k V_{2k}(q) = \frac{\eta^{(k)}(\tau)}{\eta(\tau)} = \scB_k( (\log \eta)^{(1)}(\tau), \dots, (\log \eta)^{(k)}(\tau)).
\]
Since
\[
	\frac{\dd}{\dd \tau} \log \eta(\tau) = \frac{2\pi i}{24} E_2(q)
\]
and the ring $\C[E_2(q), E_4(q), E_6(q)]$ is closed under the differentiation by the Ramanujan's identities:
\[
	\frac{1}{2\pi i} \frac{\dd E_2}{\dd \tau} = \frac{E_2^2 - E_4}{12}, \qquad \frac{1}{2\pi i} \frac{\dd E_4}{\dd \tau} = \frac{E_2 E_4 - E_6}{3}, \qquad \frac{1}{2\pi i} \frac{\dd E_6}{\dd \tau} = \frac{E_2 E_6 - E_4^2}{2},
\]
it follows that
\[
	V_{2k}(q) \in \C[E_2(q), E_4(q), E_6(q)],
\]
recovering the result by Berndt--Yee~\cite{BerndtYee2003}. While some additional argument is required to fully derive the polynomial $\phi_V(\lambda)$ of Amdeberhan--Ono--Singh, the above calculation already yields an explicit formula of similar type as follows.

\begin{theorem}\label{thm:V-even}
	For a positive integer $k$, we have
	\[
		V_{2k}(q) = 24^k \scB_k(F_1(q), F_2(q), \dots, F_k(q)),
	\]
	where we define $F_1(q) = \frac{1}{24} E_2(q)$, and for $j \ge 2$, 
	\[
		F_j(q) \coloneqq - \sum_{n=1}^\infty n^{j-1} \sigma_1(n) q^n.
	\]
\end{theorem}

Although our discussion so far focuses on the even-index case $V_{2k}(q)$, a natural generalization to all $k$, regardless of parity, emerges when viewed from the perspective of Jacobi forms. Indeed, by fully exploiting the Jacobi triple product, we can derive a similar result for
\[
	V_k(q) \coloneqq \frac{\sum_{n \in \Z} (-1)^n (6n+1)^k q^{\frac{n(3n+1)}{2}}}{\sum_{n \in \Z} (-1)^n q^{\frac{n(3n+1)}{2}}}.
\]
As observed by Amdeberhan--Griffin--Ono~\cite{AmdeberhanGriffinOno2025-pre}, the traces of partition Eisenstein series discussed here reflect a deeper structure arising from underlying Jacobi forms.

\begin{theorem}\label{thm:V-general}
	For a positive integer $k$, we have
	\[
		V_k(q) = 6^k \scB_k(G_1(q), G_2(q), \dots, G_k(q)),
	\]
	where we define
	\[
		G_1(q) = \frac{1}{6} + \sum_{n=1}^\infty \sum_{d \mid n} \left(\frac{-3}{d}\right) q^n,
	\]
	and for $j \ge 2$,
	\[
		G_j(q) = \begin{cases}
			\displaystyle{- \sum_{n=1}^\infty \sum_{\substack{d \mid n \\ n/d \not\equiv 0\ (3)}} d^{j-1} q^n} &\text{if $j$ is even},\\
			\displaystyle{\sum_{n=1}^\infty \sum_{d \mid n} \left(\frac{-3}{n/d}\right) d^{j-1} q^n} &\text{if $j \ge 3$ is odd},
		\end{cases}
	\]
	where $\left(\frac{\cdot}{\cdot}\right)$ is the Kronecker symbol.
\end{theorem}

\begin{proof}
	By applying Fa\`{a} di Bruno's formula to Jacobi's triple product
	\[
		\Theta_q(z) = \sum_{n \in \Z} (-1)^n q^{\frac{n(3n+1)}{2}} \zeta^{6n+1} = \zeta \prod_{m=1}^\infty (1- q^{3m}) (1-q^{3m-1}\zeta^6) (1-q^{3m-2} \zeta^{-6}),
	\]
	where $\zeta = e^{2\pi i z}$, we have
	\[
		(2\pi i)^k V_k(q) = \lim_{z \to 0} \frac{\Theta_q^{(k)} (z)}{\Theta_q(z)} = \scB_k((\log \Theta_q)^{(1)}(0), \dots, (\log \Theta_q)^{(k)}(0)).
	\]
	Using the infinite product expression, we find
	\[
		\frac{\dd}{\dd z} \log \Theta_q(z) = 2\pi i \bigg(1 - 6 \sum_{m, n \ge 1} q^{(3m-1)n} \zeta^{6n} + 6 \sum_{m, n \ge 1} q^{(3m-2)n} \zeta^{-6n} \bigg).
	\]
	By repeatedly differentiating this expression and then taking the limit $z \to 0$ (i.e., $\zeta \to 1$), we arrive at the desired result.
\end{proof}

For even $k$, \cref{thm:V-even} and \cref{thm:V-general} provide two different expressions for the same quantity. For example, in the case $k=2$, we have
\[
	V_2(q) = 36 (G_1(q)^2 + G_2(q)) = 24 F_1(q) = E_2(q).
\]

Finally, we note that a similar observation applies to another function,
\[
	U_{2k}(q) \coloneqq \frac{\sum_{n=0}^\infty (-1)^n (2n+1)^{2k+1} q^{\frac{n(n+1)}{2}}}{\sum_{n=0}^\infty (-1)^n (2n+1) q^{\frac{n(n+1)}{2}}},
\]
which was also studied by Amdeberhan--Ono--Singh~\cite{AmdeberhanOnoSingh2024-pre}. By focusing on the well-known identity
\[
	\sum_{n=0}^\infty (-1)^n (2n+1) q^{\frac{n(n+1)}{2}} = \prod_{m=1}^\infty (1-q^m)^3,
\]
the corresponding result follows in much the same way. Compared with \cref{thm:V-even}, this is a remarkably unified expression.

\begin{theorem}
	For a positive integer $k$, we have
	\[
		U_{2k}(q) = 8^k \scB_k(3F_1(q), 3F_2(q), \dots, 3F_k(q)),
	\]
	where $F_j(q)$ is defined as in \cref{thm:V-even}.
\end{theorem}

\section{MacMahon's partition functions}

In 1920, MacMahon~\cite{MacMahon1920} extended the classical divisor sums from the perspective of partition theory. For a positive integer $k$, \emph{MacMahon's partition function} $M_k(n)$ is defined by the generating series
\[
	A_k(q) \coloneqq \sum_{n=1}^\infty M_k(n) q^n = \sum_{0 < m_1 < m_2 < \cdots < m_k} \frac{q^{m_1 + m_2 + \cdots + m_k}}{(1-q^{m_1})^2 (1-q^{m_2})^2 \cdots (1-q^{m_k})^2}.
\]
A simple observation, essentially the same as one found in MacMahon's original work, shows that the following identity holds:
\[
	\mathcal{A}(X) \coloneqq 1 + \sum_{k=1}^\infty A_k(q) X^k = \prod_{m=1}^\infty \left(1 + \frac{q^m}{(1-q^m)^2} X\right).
\]
Applying the Fa\`{a} di Bruno formula to this identity at $X=0$, we obtain the following, which provides a new proof of the quasi-modularity of $A_k(q)$.

\begin{theorem}\label{thm:Ak-quasi}
	For a positive integer $k$, we have
	\[
		A_k(q) = \frac{1}{k!} \scB_k(H_1(q), H_2(q), \dots, H_k(q)),
	\]
	where
	\begin{align*}
		H_j(q) &= \frac{(-1)^{j-1}(j-1)!}{(2j-1)!} \sum_{l=1}^{2j} t(j,l) L_l(q),\\
		L_l(q) &= \sum_{n=1}^\infty \sigma_{l-1}(n) q^n,
	\end{align*}
	and $t(j,l) \in \Z$ are the central factorial numbers (see OEIS~\cite[A008955]{OEIS}) defined by
	\[
		\sum_{l=1}^{2j} t(j, l) x^{l-1} = \prod_{|i| < j} (x-i) = (2j-1)! \binom{x+j-1}{2j-1}.
	\]
\end{theorem}

\begin{proof}
	The infinite product expression implies that
	\[
		\frac{\dd}{\dd X} \log \mathcal{A}(X) = - \sum_{m=1}^\infty \sum_{n=1}^\infty (-1)^n \frac{q^{mn}}{(1-q^m)^{2n}} X^{n-1},
	\]
	and hence
	\[
		H_j(q) = \lim_{X \to 0} \frac{\dd^j}{\dd X^j} \log \mathcal{A}(X) = (-1)^{j-1} (j-1)!\sum_{m=1}^\infty \frac{q^{jm}}{(1-q^m)^{2j}}.
	\]
	Applying the binomial theorem, we find that
	\[
		\frac{q^{jm}}{(1-q^m)^{2j}} = \sum_{n=1}^\infty \binom{n+j-1}{2j-1} q^{mn} = \frac{1}{(2j-1)!} \sum_{n=1}^\infty \sum_{l=1}^{2j} t(j,l) n^{l-1} q^{mn},
	\]
	which leads to the desired result.
\end{proof}

\begin{corollary}
	For a positive integer $k$, we have $A_k(q) \in \Q[E_2(q), E_4(q), E_6(q)]$, that is, $A_k(q)$ is a quasi-modular form.
\end{corollary}

\begin{proof}
	Since $\prod_{|i| < j} (x-i)$ is an odd polynomial in $x$, it follows that $t(j,l) = 0$ whenever $l$ is odd. Therefore, $H_j(q) \in \Q[E_2(q), E_4(q), \dots, E_{2j}(q)]$. The fact that the polynomial ring $\Q[E_k(q): k \in 2\Z_{>0}]$ coincides with $\Q[E_2(q), E_4(q), E_6(q)]$ is a fundamental property of quasi-modualr forms (see, for example, Kaneko--Zagier~\cite{KanekoZagier1995}).
\end{proof}

The quasi-modularity of MacMahon's function $A_k(q)$ was originally proved by Andrews--Rose~\cite{AndrewsRose2013}, and an alternative proof based on an explicit generating function expression was later given by Bachmann~\cite{Bachmann2024}. Our proof may be seen as more classical in spirit, as it directly applies the Fa\`{a} di Bruno formula to MacMahon's original observation, followed by a straightforward use of the binomial theorem. 

We remark that Bachmann's explicit formula can be rewritten in the form
\[
	A_k(q) = \Lambda_k(L_2(q), L_4(q), \dots, L_{2k}(q)),
\]
where $\Lambda_k(X_1, \dots, X_k)$ is a polynomial defined as a slight modification of the Bell polynomial generating function \eqref{def:Bell-polynomial}, given by
\[
	\sum_{k=0}^\infty \Lambda_k(X_1, \dots, X_k) t^{2k} = \prod_{j=1}^\infty \exp \left(\frac{2(-1)^{j-1}}{(2j)!} \left(2 \arcsin \frac{t}{2}\right)^{2j} X_j \right).
\]
In joint work with Kang and Shin~\cite{KangMatsusakaShin2025}, we showed that this same polynomial $\Lambda_k$ also describes MacMahon's variants and the generalizations introduced by Rose~\cite{Rose2015}. This allowed us to prove quasi-modularity in a broader setting. The proof of \cref{thm:Ak-quasi} presented here applies equally well to Rose's generalized case. Moreover, in the same joint article with Kang and Shin, we suggested possible connections with Lehmer's theorem, as well as with related work by the author and Shibukawa~\cite{MatsusakaShibukawa2024}. In this context, the Fa\`{a} di Bruno formula seems to provide a natural explanation to these observations.

This method using Fa\`{a} di Bruno's formula is also applicable to another family of MacMahon-type functions recently introduced by Amdeberhan--Andrews--Tauraso~\cite{AmdeberhanAndrewsTauraso2025} and further studied by Nazaroglu--Pandey--Singh~\cite{NazarogluPandeySingh2025-pre}. For positive integers $k, t, r$, we define
\[
	A_{k,t,r}(a; q) \coloneqq \sum_{0 < m_1 < m_2 < \cdots < m_k} \frac{q^{r(m_1 + m_2 + \cdots + m_k)}}{(1+aq^{m_1} + q^{2m_1})^t \cdots (1+aq^{m_k} + q^{2m_k})^t}.
\]
While $a$ is treated as an integer in this context, the argument does not require it. We note that $A_{k,1,1}(-2; q) = A_k(q)$. Applying the Fa\`{a} di Bruno formula, we obtain the following.

\begin{theorem}
	For positive integers $k, t, r$, we have
	\[
		A_{k,t,r}(q) = \frac{1}{k!} \scB_k(H_{1,t,r}(a;q), H_{2,t,r}(a;q), \dots, H_{k,t,r}(a;q)),
	\]
	where we define
	\[
		H_{j,t,r}(a; q) = (-1)^{j-1} (j-1)! \sum_{m=1}^\infty \frac{q^{rjm}}{(1+aq^m+q^{2m})^{tj}}.
	\]
\end{theorem}

\begin{proof}
	The proof is entirely analogous to that of \cref{thm:Ak-quasi} and is therefore omitted.
\end{proof}

For example, the proof of quasi-modularity given by Nazaroglu--Pandey--Singh~\cite{NazarogluPandeySingh2025-pre} can be reorganized in the following way, by reducing the problem to the quasi-modularity of $H_{j,t,r}(a;q)$. 
This does not cover the full scope of their results, but we present one representative example here to illustrate the main idea. For the definition of quasi-modular forms and further details, we refer the reader to their article.

\begin{corollary}
	For positive integers $k, t$, and $a=1$, $A_{k,t,t}(1;q)$ is a quasi-modular form of level $3$.
\end{corollary}

\begin{proof}
	As shown in \cite[Section 3]{NazarogluPandeySingh2025-pre},
	\[
		\sum_{m=1}^\infty \frac{q^{tjm}}{(1+q^m+q^{2m})^{tj}}
	\]
	is a linear combination of $1, L_l(q), L_l(q^3)$, and 
	\[
		\sum_{n=1}^\infty \sum_{d \mid n} \left(\frac{-3}{d}\right) d^l q^n
	\]
	with even $l$. Hence, this is a quasi-modular form of level $3$, that is, the quasi-modularity of $A_{k,t,t}(1; q)$ follows.
\end{proof}

\section{Reciprocal sums of parts in integer partitions}\label{sec:SRP}

Finally, we conclude with some remarks on recent work by Byungchan Kim and Eunmi Kim~\cite{KimKim2025}, focusing on what the Fa\`{a} di Bruno formula reveals about their topic of study. For a positive integer $n$, let $\mathcal{D}_n$ denote the set of partitions $\lambda = (\lambda_1, \dots, \lambda_r) \vdash n$ into distinct parts, that is, $\lambda_i \neq \lambda_j$ whenever $i \neq j$. We define the map $\mathrm{srp}: \mathcal{D}_n \to \Q$ by 
\[
	\mathrm{srp}(\lambda) = \sum_{j=1}^{\ell(\lambda)} \frac{1}{\lambda_j}.
\]
Graham~\cite[Theorem 1]{Graham1963} proved that for $n > 77$, there exists a partition $\lambda \in \mathcal{D}_n$ such that $\mathrm{srp}(\lambda) = 1$. Following this result, they studied a more general family of moments defined by
\[
	s_k(n) \coloneqq \sum_{\lambda \in \mathcal{D}_n} \mathrm{srp}(\lambda)^k
\]
for positive integers $k$, and, in collaboration with Bringmann~\cite{BringmannKimKim2024-pre, BringmannKimKim2025-pre}, focused on their asymptotic behavior and modular aspects.

The main focus here is the generating function identity
\[
	\sum_{n=1}^\infty s_k(n) q^n = \lim_{\zeta \to 1} \left(\zeta \frac{\dd}{\dd \zeta}\right)^k \prod_{m=1}^\infty (1 + \zeta^{1/m} q^m)
\]
noted in~\cite{BringmannKimKim2024-pre}. They further introduced the functions
\[
	g_k(q) \coloneqq \sum_{n=1}^\infty \frac{1}{n^k} \sum_{d \mid n} (-1)^{d-1} d^{2k-1} q^n = \lim_{\zeta \to 1} \left(\zeta \frac{\dd}{\dd \zeta}\right)^k \sum_{m=1}^\infty \log(1+\zeta^{1/m} q^m)
\]
to establish that
\[
	\prod_{m=1}^\infty \frac{1}{1+q^m} \cdot \sum_{n=1}^\infty s_k(n) q^n \in \Z[g_1(q), g_2(q), \dots, g_k(q)].
\]
However, their proof followed an inductive structure, and as in the various examples discussed in the previous sections, the explicit description of the resulting polynomials was considered to be complicated. For these polynomials as well, an explicit formula can be derived immediately from the Fa\`{a} di Bruno formula.

\begin{theorem}
	For a positive integer $k$, we have
	\[
		\prod_{m=1}^\infty \frac{1}{1+q^m} \cdot \sum_{n=1}^\infty s_k(n) q^n = \scB_k(g_1(q), g_2(q), \dots, g_k(q)).
	\]
\end{theorem}

\begin{proof}
	Applying Fa\`{a} di Bruno's formula to
	\[
		\theta_q(z) = \prod_{m=1}^\infty (1+\zeta^{1/m} q^m)
	\]
	at $z = 0$ with $\zeta = e^{2\pi iz}$, and using the relation $\zeta \frac{\dd}{\dd \zeta} = \frac{1}{2\pi i} \frac{\dd}{\dd z}$, we obtain the desired identity.
\end{proof}

\bibliographystyle{amsalpha}
\bibliography{References} 

\end{document}